\documentclass[11pt]{article}

\author{}
\date{}
\usepackage{graphicx}
\usepackage{tabularx}
\usepackage[hidelinks]{hyperref}
\usepackage{amsthm} 
\usepackage{amsmath} 
\usepackage{amsfonts}
\usepackage{amssymb}
\usepackage{tikz}
\usepackage{enumitem}
\usepackage{relsize}
\usepackage{subcaption}
\usepackage{bbm}
\usepackage{todonotes}
\usetikzlibrary{decorations.pathreplacing}
\tikzstyle{vertex}=[circle,fill=black!100,text=white,inner sep=0.8mm]
\tikzstyle{point}=[circle,fill=black,inner sep=0.1mm]

\usepackage{thmtools} 
\usepackage{thm-restate}

\DeclareMathOperator{\cw}{\textup{cw}}
\DeclareMathOperator{\vs}{\textup{vs}}
\DeclareMathOperator{\lev}{lvl}

\usepackage{color}
\usepackage{soul}

\newcommand{\cT}{\mathcal{T}}
\newcommand{\cP}{\mathcal{P}}

\newtheorem{lemma}{Lemma}

\newtheorem{theorem}[lemma]{Theorem}

\theoremstyle{definition}

\newtheorem{remark}[lemma]{Remark}

\newtheorem{question}{Question}

\title{The treewidth and pathwidth of graph unions}

\author{Bogdan Alecu\thanks{School of Computing, University of Leeds, UK. Email: \texttt{b.alecu@leeds.ac.uk}} \and 
	Vadim Lozin\thanks{Mathematics Institute, University of Warwick, Coventry, CV4 7AL, UK. Email: \texttt{V.Lozin@warwick.ac.uk}} \and
	Daniel A. Quiroz\thanks{Instituto de Ingenier\'ia Matem\'atica-CIMFAV, Universidad de Valparaiso, Chile. Email: \texttt{daniel.quiroz@uv.cl}} \and
	Roman Rabinovich\thanks{ArangoDB Inc. Email: \texttt{roman.rabinovich@arangodb.com}} \and 
	Igor Razgon\thanks{Department of Computer Science and Information Systems, Birkbeck University of London. Email: \texttt{igor@dcs.bbk.ac.uk}} \and 
	Viktor Zamaraev\thanks{Department of Computer Science, University of Liverpool, Liverpool, UK. Email: \texttt{viktor.zamaraev@liverpool.ac.uk}}}

\begin{document}
	\maketitle

\begin{abstract}
	Given two $n$-vertex graphs $G_1$ and $G_2$ of bounded treewidth,
is there an $n$-vertex graph $G$ of bounded treewidth having subgraphs 
isomorphic to $G_1$ and $G_2$?
Our main result is a negative answer to this question, in a strong sense:
we show that the answer is \emph{no} even if $G_1$ is a binary tree and $G_2$ is a ternary
tree. We also provide an extensive study of cases where
such `gluing' is possible.
In particular,
we prove that if $G_1$ has treewidth $k$ and $G_2$ has pathwidth $\ell$, then there is an $n$-vertex graph of treewidth at most $k + 3 \ell + 1$ containing both $G_1$ and $G_2$ as subgraphs.


\end{abstract}

\section{Introduction}
\paragraph{The main results and  motivation.} 
In this paper we consider the following question:
given two $n$-vertex graphs $G_1$ and $G_2$ of bounded treewidth,
is there an $n$-vertex graph $G$ of bounded treewidth having subgraphs 
isomorphic to $G_1$ and $G_2$?
Our main result is a negative answer to this question, in a strong sense:
we show that the answer is \emph{no} even if $G_1$ is a binary tree and $G_2$ is a ternary
tree. The proof is deep in the sense that it uses a number-theoretic argument that
is not evident from the statement of the question.
In addition to the main result, we also provide an extensive study of cases where
such `gluing' is possible.
In particular,
we prove that if $G_1$ has treewidth $k$ and $G_2$ has pathwidth $\ell$, 
then they can be united into a graph of treewidth at most $k + 3 \ell + 1$.

The motivation for this research comes from network design.
Many real-world systems can be modeled as networks, where the entities of the system are represented by the network's nodes, and interactions between the entities are modelled by the edges. In many complex systems, the entities can exhibit multiple types of interactions.
For example, nodes of a transportation network can be connected via two types of edges,
one type representing train connections and the other type representing flight connections. 
The abstraction of multilayer networks is a means of modelling complex systems more precisely,
by capturing the different ``layers'' of information \cite{KABGMP14}.
A natural direction in the research of multilayer networks is investigating the extent to which 
good properties of individual layers can be exploited in the treatment of the network formed by
combining the layers.

Consider the scenario of a multilayered network where we are provided
with layers of a specified topology close to that of a tree. We want to find
out whether these layers can be united together into a single network that
is also close to a tree. 
Enright, Meeks, and Ryan \cite{EMR17} consider several concrete applications
where the above pattern of layers occurs. 
 They show that when the layers are combined adversarially, good properties of individual layers are lost (unless
the structure of the layers is severely restricted).
Moreover, the same holds with high probability even if there are just two layers, each isomorphic to a path, that are combined randomly \cite{EMPS21}.
Our result naturally continues this line of research by considering the question from the designer's
perspective. In other words, we study the situation when the resulting network is not \emph{imposed}
on us by a (possibly random) adversary, but rather we have the power to decide how the  network will be formed from the layers. 
Our main result shows that even in this ideal case, we cannot, in general, avoid
arbitrarily large treewidth in the resulting network.


On the other hand, we show that in some special cases the two layers can always be combined in a network of bounded treewidth. In particular, our result concerning the gluing of a graph of a bounded treewidth and a graph of a bounded pathwidth
significantly generalises previous results stated in \cite{M05}, \cite{CF12}, and \cite{Q17}. 

\paragraph{Formal statements.}
Let $n$ be a positive integer. We denote by $[n]$ the set $\{1, \dots, n\}$, and by $S_n$
the symmetric group of all permutations of $[n]$.
For a permutation $\varphi \in S_n$ and sets $U \subseteq [n]$ and $E \subseteq {[n] \choose 2}$,
we write $\varphi(U) := \{\varphi(i) : i \in U\}$ and $\varphi(E) := \{\{\varphi(i), \varphi(j)\} : \{i, j\} \in E\}$.
Given two graphs $G_1 = ([n], E_1)$ and $G_2 = ([n], E_2)$, the {\em union of $G_1$ and $G_2$ along $\varphi$} is the graph $([n], \varphi(E_1) \cup E_2)$. 
A \emph{gluing of $G_1$ and $G_2$} is the union of $G_1$ and $G_2$ along some permutation.
We may think of this operation as first relabeling the vertices of $G_1$ according to $\varphi$
and then taking the union of the resulting graph with $G_2$ (see Figure~\ref{fig-tree-union} for illustration).

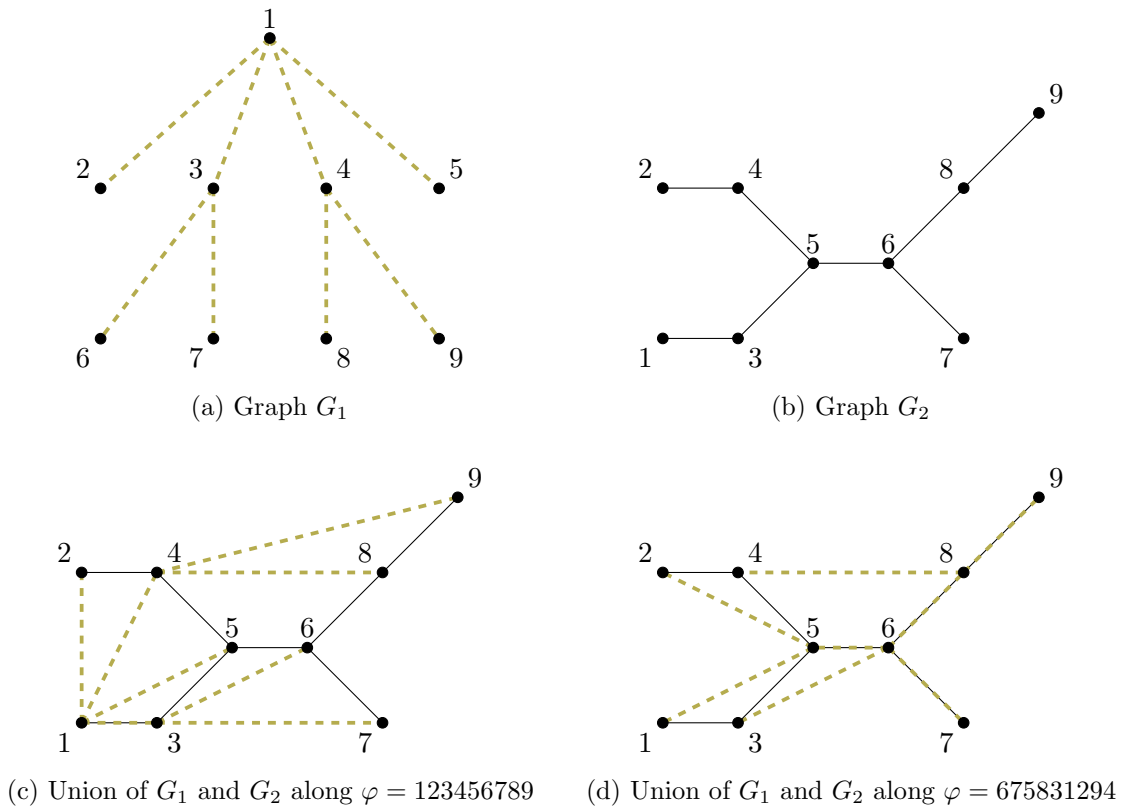
\begin{figure}[ht]
	\centering
	\begin{subfigure}[t]{0.49\linewidth}
		\centering
		\begin{tikzpicture}[scale=1, transform shape]
			
			\draw[olive!70, line width=0.5mm, dashed]
			(0, 4) -- (-2.25,2)
			(0, 4) -- (-0.75,2)
			(0, 4) -- (0.75,2)
			(0, 4) -- (2.25,2)
			(-0.75,2) -- (-2.25,0)
			(-0.75,2) -- (-0.75,0)
			(0.75,2) -- (0.75,0)
			(0.75,2) -- (2.25,0);
			
			\filldraw 
			(0,4) circle (2pt) node[above] {$1$}
			(-2.25,2) circle (2pt) node[above left] {$2$}
			(-0.75,2) circle (2pt) node[above left] {$3$}
			(0.75,2) circle (2pt) node[above right] {$4$}
			(2.25,2) circle (2pt) node[above right] {$5$}
			(-2.25,0) circle (2pt) node[below left] {$6$}
			(-0.75,0) circle (2pt) node[below left] {$7$}
			(0.75,0) circle (2pt) node[below right] {$8$}
			(2.25,0) circle (2pt) node[below right] {$9$};

		\end{tikzpicture}
		\captionsetup{justification=centering}
		\caption{Graph $G_1$}
		\label{middlepoint}	
	\end{subfigure}	
	\begin{subfigure}[t]{0.49\linewidth}
		\centering
		\begin{tikzpicture}[scale=1, transform shape]
			
			\filldraw 
			(0,-1) circle (2pt) node[below left] {$1$}
			(0,1) circle (2pt) node[above left] {$2$}
			(1,-1) circle (2pt) node[below right] {$3$}
			(1,1) circle (2pt) node[above right] {$4$}
			(2,0) circle (2pt) node[above] {$5$}
			(3,0) circle (2pt) node[above] {$6$}
			(4,-1) circle (2pt) node[below left] {$7$}
			(4,1) circle (2pt) node[above left] {$8$}
			(5,2) circle (2pt) node[above right] {$9$};
			
			\draw 
			(0, -1) -- (1, -1)
			(0, 1) -- (1, 1)
			(1, -1) -- (2, 0)
			(1, 1) -- (2, 0)
			(2, 0) -- (3, 0)
			(3, 0) -- (4, 1)
			(3, 0) -- (4, -1)
			(4, 1) -- (5, 2);

		\end{tikzpicture}
		\captionsetup{justification=centering}
		\caption{Graph $G_2$}
		\label{middlepointareas}	
	\end{subfigure}
	
	\bigskip
	
	\begin{subfigure}[t]{0.49\linewidth}
		\centering
		\begin{tikzpicture}[scale=1, transform shape]
			
			\draw (0, -1) -- (1, -1);
			
			\draw[olive!70, line width=0.5mm, dashed]     
			(0, -1) -- (1, -1)
			(0, -1) -- (1, 1)
			(0, -1) -- (0, 1)
			(0, -1) -- (2, 0)
			(1, -1) -- (3, 0)
			(1, -1) -- (4, -1)
			(1, 1) -- (4, 1)
			(1, 1) -- (5, 2);
			
			\draw 
			(0, 1) -- (1, 1)
			(1, -1) -- (2, 0)
			(1, 1) -- (2, 0)
			(2, 0) -- (3, 0)
			(3, 0) -- (4, 1)
			(3, 0) -- (4, -1)
			(4, 1) -- (5, 2);
			
			\filldraw 
			(0,-1) circle (2pt) node[below left] {$1$}
			(0,1) circle (2pt) node[above left] {$2$}
			(1,-1) circle (2pt) node[below right] {$3$}
			(1,1) circle (2pt) node[above right] {$4$}
			(2,0) circle (2pt) node[above] {$5$}
			(3,0) circle (2pt) node[above] {$6$}
			(4,-1) circle (2pt) node[below left] {$7$}
			(4,1) circle (2pt) node[above left] {$8$}
			(5,2) circle (2pt) node[above right] {$9$};
			
		\end{tikzpicture}
		\captionsetup{justification=centering}
		\caption{\small Union of $G_1$ and $G_2$ along $\varphi = 123456789$}
		\label{middlepoint}	
	\end{subfigure}
	\begin{subfigure}[t]{0.49\linewidth}
		\centering
		\begin{tikzpicture}[scale=1, transform shape]

			\draw
			(2, 0) -- (3, 0)
			(3, 0) -- (4, 1)
			(3, 0) -- (4, -1)
			(4, 1) -- (5, 2);
			
			\draw[olive!70, line width=0.5mm, dashed]     
			(3,0) -- (4,-1)
			(3,0) -- (2,0)
			(3,0) -- (4,1)
			(3,0) -- (1,-1)
			(2,0) -- (0,-1)
			(2,0) -- (0,1)
			(4,1) -- (5,2)
			(4,1) -- (1,1);
			
			\draw 
			(0,-1) -- (1,-1)
			(0, 1) -- (1, 1)
			(1, -1) -- (2, 0)
			(1, 1) -- (2, 0);
			
			\filldraw 
			(0,-1) circle (2pt) node[below left] {$1$}
			(0,1) circle (2pt) node[above left] {$2$}
			(1,-1) circle (2pt) node[below right] {$3$}
			(1,1) circle (2pt) node[above right] {$4$}
			(2,0) circle (2pt) node[above] {$5$}
			(3,0) circle (2pt) node[above] {$6$}
			(4,-1) circle (2pt) node[below left] {$7$}
			(4,1) circle (2pt) node[above left] {$8$}
			(5,2) circle (2pt) node[above right] {$9$};
			
		\end{tikzpicture}
		\captionsetup{justification=centering}
		\caption{Union of $G_1$ and $G_2$ along $\varphi = 675831294$}
		\label{middlepoint2}	
	\end{subfigure}		
	
	\caption{The union of two trees along different permutations. }
	\label{fig-tree-union}
	
\end{figure}
Our main theorem is the following.

\begin{restatable}{theorem}{main}
	\label{thm:main}
	For any $c > 0$, there exists $n \in \mathbb N$, and $n$-vertex trees $T_1$ and $T_2$ such that any gluing of $T_1$ and $T_2$ has treewidth at least $c$.
\end{restatable} 

On the positive side, we demonstrate that a graph of bounded treewidth and a graph of bounded pathwidth can be
glued into a graph of bounded treewidth. 

\begin{restatable}{theorem}{twpw}
	\label{th:tw-pw}
	Let $G_1$ be an n-vertex graph of treewidth at most $k$, and $G_2$ an $n$-vertex graph of pathwidth at most $\ell$. Then there exists a gluing of $G_1$ and $G_2$ that has treewidth at most $k+3\ell+1$.
\end{restatable}

In the case of more restricted classes, we can provide an even more refined view as described in the two statements below. 
\begin{restatable}{lemma}{twvc}
	\label{obs:tw-vc}
	Any gluing of 
	an $n$-vertex graph of vertex cover number at most $k$ and
	an $n$-vertex graph of treewidth at most $t$ has treewidth at most $k+t$.
\end{restatable}

\begin{restatable}{lemma}{pathwidth}
	\label{lem:pathwidth}
	Let $G_1$ and $G_2$ be $n$-vertex graphs of pathwidth $k$ and $t$ respectively.
	Then there is a gluing of $G_1$ and $G_2$ of pathwidth at most $k+t$.
\end{restatable}

\noindent
In particular, Lemma \ref{lem:pathwidth} implies that two caterpillars can always be glued into a graph of treewidth
at most 2.

Finally, our results can be stated in terms of graph classes.
In particular, for two classes of graphs $\mathcal{X}$ and $\mathcal{Y}$ a \emph{gluing of $\mathcal{X}$ and $\mathcal{Y}$}
is a minimal class which, for every pair of $n$-vertex graphs $G_1 \in \mathcal{X}$ and $G_2 \in \mathcal{Y}$,
contains a gluing of $G_1$ and $G_2$ of minimum possible treewidth. A graph parameter is said to be bounded
for a class $\mathcal{X}$ if there exists a constant $c$ such that for every graph in $\mathcal{X}$ the graph
parameter does not exceed $c$; otherwise the graph parameter is said to be unbounded in $\mathcal{X}$.
\begin{theorem}
	\label{thm:characterization}
	A gluing of two minor-closed classes of graphs $\mathcal{X}$ and $\mathcal{Y}$ has bounded treewidth if and only if
	both $\mathcal{X}$ and $\mathcal{Y}$  have bounded treewidth and one of them has bounded pathwidth.
\end{theorem}
\begin{proof}
	Clearly, if $\mathcal{X}$ or $\mathcal{Y}$ has unbounded treewidth, any gluing of the two classes has unbounded
	treewidth. Since both classes are minor-closed, a gluing of $\mathcal{X}$ and $\mathcal{Y}$ 
	has bounded treewidth if and only if at least one of the classes has bounded 
	pathwidth: this follows from Theorems \ref{th:tw-pw} and \ref{thm:main}, 
	and a result of Robertson and Seymour saying that a minor-closed class has bounded pathwidth
	 if and only if the class excludes a forest \cite{RS83}.   
\end{proof}

The rest of paper is organised as follows. 
In Section \ref{sec:prelim}, we introduce some standard preliminaries.
In Section \ref{sec:counterexample}, we prove that there are pairs of trees that cannot
be glued into a graph of small treewidth.
In Section \ref{sec:positive}, we show that certain graphs of bounded treewidth
can be glued into a graph of bounded treewidth. 
In Section \ref{sec:conclusion}, we conclude the paper with some open questions.

\section{Preliminaries}\label{sec:prelim}
For a graph $G$ we denote by $V(G)$ and $E(G)$ the vertex set and the edge set of $G$ respectively.
A vertex $v \in V(G)$ is a \emph{neighbour} of another vertex $u \in V(G)$ if $\{u,v\}$ is an edge
of $G$. The \emph{neighbourhood} of $u$, denoted $N_G(u)$, is the set of all neighbours of $u$. The \emph{degree} of $u$ is the number of its neighbours.
For a vertex set $U \subseteq V(G)$ the subgraph of $G$ induced by $U$ is denoted by $G[U]$.
A set of vertices $C \subseteq V(G)$ is a \emph{vertex cover} of $G$ if every edge of $G$ is incident with a vertex in $C$. 
The \emph{vertex cover number} of $G$ is the minimum number of vertices in a vertex cover of $G$. 
As usual, we will denote by $K_s$ a complete graph on $s$ vertices and by $K_{1,s}$ a star with $s$ leaves. A \emph{caterpillar} is a tree that becomes a path upon removal of all its leaves.

A \emph{tree decomposition} of a graph $G$ is a pair $\left(\cT,(X_i)_{i\in V(\cT)}\right)$, 
where $\cT$ is a tree and $X_i\subseteq V(G)$ for each $i\in V(\cT)$, such that
\begin{enumerate}
	\item[(I)] $\bigcup_{i\in V(\cT)}X_i=V(G)$;
	\item[(II)] for every edge $\{u,v\}\in E(G)$, there is a $i\in V(\cT)$ such that
	$u,v\in X_i$; and
	\item[(III)] for every $v \in V(G)$ the subgraph $\cT_v$ of $\cT$
	induced by $\{i \in V(\cT) ~|~ v \in X_{i} \}$ is connected, i.e. $\cT_v$ is a tree.
\end{enumerate}
So as to avoid confusion with the vertices of $G$, we say that the elements of $V(\cT)$ are the \emph{nodes} of $\cT$. For a node $i$ we say that the corresponding set $X_i$ is the \emph{bag} of $i$. The width of the tree decomposition $\left(\cT,(X_i)_{i\in V(\cT)}\right)$
is $\max_{i\in V(\cT)}|X_i|-1$. The \emph{treewidth of $G$} is the smallest width of a tree decomposition of $G$.

Alternatively, treewidth can be defined via partial $k$-trees. A \emph{$k$-tree} is a graph that can be obtained by starting with $K_k$ and repeatedly adding vertices and connecting them to a clique of size $k$. A {\em partial $k$-tree} is a (not necessarily induced) subgraph of a $k$-tree. The treewidth of a graph $G$ is equal to the least $k$ such that $G$ is a partial $k$-tree.

A \emph{path decomposition} of a graph $G$ is a tree decomposition $\left(\cP,(X_i)_{i\in V(T)}\right)$ 
in which the tree $\cP$ is a path. The {\em pathwidth} of $G$ is equal to the smallest width of any path decomposition of $G$. 

Pathwidth admits a characterisation via vertex separation number that we will employ in one of the proofs.
A \emph{layout} of a graph is a linear ordering of its vertices.
Let $G=(V,E)$ be a graph and let $\pi$ be a layout of $G$. The \emph{vertex separation number of $G$ with respect to $\pi$} is defined as 
$$
\vs_{\pi}(G) = \max_{v\in V}|\{w\in V: \exists x\in N_G(w) \mbox{ such that } \pi(w) < \pi(v)\le \pi(x) \}|.
$$ 
The \emph{vertex separation number} $\vs(G)$ of $G$ is the minimum of $\vs_{\pi}(G)$ over all possible layouts $\pi$ of $G$.

\begin{theorem}[Kinnersley \cite{K92}] \label{thm:pathwidth-vs-number}
	The vertex separation number of a graph equals its pathwidth.
\end{theorem}

\section{Unions of trees have unbounded treewidth} \label{sec:counterexample}

This section is devoted to our main result, that the union of two trees can have arbitrarily large treewidth. In fact, we show that it can have arbitrarily large clique-width. Before proceeding to the proof of this, found in Section~\ref{subsec:core}, we start with some necessary preliminaries.

\subsection{Trees and cuts}

Let $G=(V,E)$ be a graph. For a vertex set $U \subseteq V$ we denote by $\overline{U}$ the set $V \setminus U$.
The partition $(U, \overline{U})$ of the vertex set of $G$ is called the \emph{$U$-cut} of $G$.
The \emph{$U$-cut-set} in $G$ is the set of edges of $G$ that have one endpoint in $U$ and the other endpoint in $\overline{U}$. The edges in the \emph{$U$-cut-set} are called the \emph{crossing edges} of the $U$-cut.
We denote by $e_{G}(U)$ the number of crossing edges of the $U$-cut in $G$.
The $U$-cut is called \emph{balanced} if $\frac{n}{3} \leq |U| \leq \frac{2n}{3}$, where $n$ is the number of vertices
in $G$.

Now let $T = (V, E)$ be a rooted tree. The \emph{level of a vertex} $v$ in $T$, denoted by $\lev(v)$, is the distance from $v$ to the root of $T$.
In particular, the level of the root is 0.
We denote by $T^v$ the subtree of $T$ rooted at $v$, and by $n_v$ the number of vertices in $T^v$. By convention, the vertices of an edge $e = (u, v)$ of $T$ are ordered so that
$\lev(u) = \lev(v) - 1$. Given a set $U \subseteq V$, we write $\mathbbm 1_U$ for the indicator function given by $\mathbbm 1_U(v) = 1$ if $v \in U$, and $0$ otherwise.

\begin{lemma}\label{lm:treeCuts}
	Let $T=(V,E)$ be a tree on $n$ vertices rooted at $r$. Let $U \subseteq V$ be a vertex set and let $e_i = (u_i, v_i)$, $i = 1, \dots, k$ be the edges of
	the $U$-cut-set. Then
	$$
	|U| = \mathbbm{1}_{U}(r) \cdot n + \sum_{i=1}^k (-1)^{\mathbbm{1}_{U}(u_i)} \cdot n_{v_i}.
	$$
\end{lemma}
\begin{proof}
	We prove the lemma by induction on $k$. For $k=1$, there is a unique edge $e_1 = (u_1, v_1)$ in the
	$U$-cut-set. It is easy to see that $u_1$ belongs to $U$ if and only if $r$ belongs to $U$. If they both belong to $U$, then $|U| = n - n_{v_1}$. Otherwise, $|U| = n_{v_1}$. In either case, $|U| = \mathbbm{1}_{U}(r) \cdot n + (-1)^{\mathbbm{1}_{U}(u_i)} \cdot n_{v_1}$, as required.
	
	Let now $k > 1$ and assume, without loss of generality, that the edge $e_k = (u_k, v_k)$
	is a \emph{minimal} edge of the $U$-cut-set, i.e., no edge of the subtree $T^{v_k}$ belongs to
	the $U$-cut-set.
	
	Suppose first that $u_k \in U$ and $v_k \in \overline{U}$. This assumption and the minimality of
	$e_k$ implies that
	$V(T^{v_k}) \subseteq \overline{U}$. By moving $V(T^{v_k})$ from one side of the cut to the other, we will
	remove exactly one edge, namely $e_k$, from the cut-set. More formally, let $U' = U \cup V(T^{v_k})$.
	Then clearly the $U'$-cut-set is equal to $\{e_1, e_2, \ldots, e_{k-1}\}$. By the induction hypothesis, and since $\mathbbm{1}_{U'}(w)=\mathbbm{1}_{U}(w)$ for every $w\in \{r, u_1, \dots, u_{k-1}\}$ we have
	$$
	|U'| = \mathbbm{1}_{U}(r) \cdot n + \sum_{i=1}^{k - 1} (-1)^{\mathbbm{1}_{U}(u_i)} \cdot n_{v_i},
	$$
	and therefore $|U|  = |U'| - n_{v_k} =  |U'| + (-1)^{\mathbbm{1}_{U}(u_k)} \cdot n_{v_k } = \mathbbm{1}_{U}(r) \cdot n + \sum_{i=1}^k (-1)^{\mathbbm{1}_{U}(u_i)} \cdot n_{v_i}$, as required.
	
	Assume now that $u_k \in \overline{U}$ and $v_k \in U$. Then, similarly to the above argument,
	$V(T^{v_k}) \subseteq U$ and if we define $U' = U \setminus V(T^{v_k})$, then
	the $U'$-cut-set is equal to $\{e_1, e_2, \ldots, e_{k-1}\}$.
	Therefore, $|U|  = |U'| + n_{v_k} = |U'| + (-1)^{\mathbbm{1}_{U}(u_k)} \cdot n_{v_k} = \mathbbm{1}_{U}(r) \cdot n + \sum_{i=1}^k (-1)^{\mathbbm{1}_{U}(u_i)} \cdot n_{v_i}$, which completes the proof.
\end{proof}

\subsection{Balanced trees}	

For a natural number $b \geq 2$, a {\em $b$-ary tree} is a rooted tree in which each vertex has at most $b$ children. For $\ell \in \mathbb N$, the {\em $\ell$-th level} of a $b$-ary tree is the set of vertices with level $\ell$. The $\ell$-th level of the tree is called \emph{last} if it is non-empty and the $(\ell+1)$-th level
is empty.
The $\ell$-th level is said to be {\em filled}, or {\em full}, if it contains $b^\ell$ vertices. In particular, if a level is filled, so are all the levels before it. 
A $b$-ary tree is {\em perfect} if all its non-empty levels are filled.
A rooted $b$-ary tree $T$ is called {\em balanced} if:
\begin{enumerate}
	\item every non-empty level of $T$, except possibly the last one, is filled;
	
	\item if $x$ and $y$ are two vertices on the same level of $T$, then $|n_x - n_y| \leq 1$.
\end{enumerate}

\begin{lemma} \label{lem:balanced-tree}
	For any integers $b \geq 2$ and $n \geq 1$, there exists a balanced $b$-ary tree with $n$ vertices.
\end{lemma}

\begin{proof}
	The statement is obvious when $n = \frac{b^\ell - 1}{b - 1}$ for some natural $\ell \geq 1$, since perfect $b$-ary trees with $\ell$ levels are balanced. 
	In general, suppose $\frac{b^\ell - 1}{b - 1} \leq n < \frac{b^{\ell + 1} - 1}{b - 1}$, and assume there is a balanced $b$-ary tree $T$ on $n$ vertices. The bounds on $n$ imply that $T$ has $\ell$ filled levels. The $(\ell+1)$-th level of $T$ is non-filled and possibly empty. We think of this final level as consisting of $b^\ell$ slots, $n - \frac{b^\ell - 1}{b - 1}$ of which are already filled with leaves. We describe how to add a leaf to one of the empty slots in this level, in order to obtain a balanced $b$-ary tree on $n + 1$ vertices.
	
	Let $x_0, x_1, \ldots, x_{\ell-1}$ be a path from the root $x_0$ to a vertex at level $\ell-1$ (i.e., the lowest filled level)
	such that for every $i \in [\ell-1]$ vertex $x_i$ is a child of $x_{i-1}$ with the minimum number of descendants.
	First, we claim that every vertex $x_i$ of the path has the minimum number of descendants among the vertices at level $i$.
	This is clearly true for $x_0$ as there is only one vertex at level 0. Assume the claim is true for $x_{i-1}$, $i \geq 1$,
	and suppose, towards a contradiction, there is a vertex $v$ at level $i$ such that $n_v < n_{x_i}$.
	The choice of $x_i$ implies that the parent $p(v)$ of $v$ is distinct from the parent $x_{i-1}$ of $x_i$.
	Since the tree is balanced and $x_i$ is the child of $x_{i-1}$ with the least number of descendants, we conclude 
	that $n_x = n_v + 1$ for every child $x$ of $x_{i-1}$, and $n_y \in \{ n_v, n_v +1 \}$ for every child $y$ of $p(v)$.
	Consequently, as $x_{i-1}$ and $p(v)$ both have $b$ children, it follows that $p(v)$ has fewer descendants 
	than $x_{i-1}$ does, which contradicts the induction assumption.
	
	Now, to complete the proof, we add the new vertex as a child of $x_{\ell-1}$. This extension of the tree increases $n_{x_i}$
	by exactly 1 for every $i = 0, 1, \ldots, \ell-1$, and does not affect the number of descendants of any other vertex in the tree.
	Since $x_i$ is a vertex with the minimum number of descendants at level $i$, it is easy to see that the balancedness property
	is preserved in the new tree.
	%
\end{proof}

Figure~\ref{fig:comp-tree-ex} provides an illustration of balanced binary and ternary trees; adding the leaves in the order given by their labels preserves balancedness at each step.

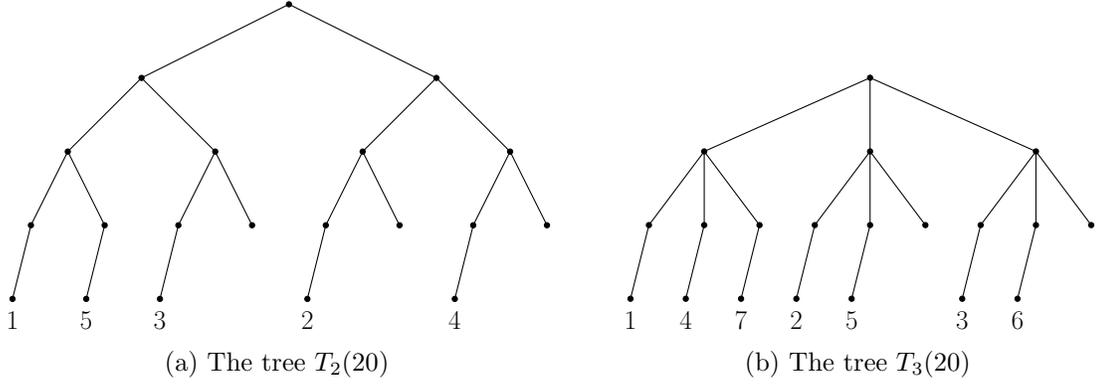
\begin{figure}[ht]
	\centering
	\begin{subfigure}[t]{0.49\linewidth}
		\centering
		\begin{tikzpicture}[scale=0.49, transform shape]
			
			\filldraw (0, 0) circle (2pt) node{};
			\draw (0, 0) -- (-4, -2);
			\draw (0, 0) -- (4, -2);
			
			\filldraw (-4, -2) circle (2pt) node{};
			\filldraw (4, -2) circle (2pt) node{};
			\draw (-4, -2) -- (-6, -4);
			\draw (-4, -2) -- (-2, -4);
			\draw (4, -2) -- (2, -4);
			\draw (4, -2) -- (6, -4);
			
			\filldraw (-6, -4) circle (2pt) node{};
			\filldraw (-2, -4) circle (2pt) node{};
			\filldraw (2, -4) circle (2pt) node{};
			\filldraw (6, -4) circle (2pt) node{};
			\draw (-7, -6) -- (-6, -4);
			\draw (-5, -6) -- (-6, -4);
			\draw (-3, -6) -- (-2, -4);
			\draw (-1, -6) -- (-2, -4);
			\draw (1, -6) -- (2, -4);
			\draw (3, -6) -- (2, -4);
			\draw (5, -6) -- (6, -4);
			\draw (7, -6) -- (6, -4);
			
			\filldraw (-7, -6) circle (2pt) node{};
			\filldraw (-5, -6) circle (2pt) node{};
			\filldraw (-3, -6) circle (2pt) node{};
			\filldraw (-1, -6) circle (2pt) node{};
			\filldraw (1, -6) circle (2pt) node{};
			\filldraw (3, -6) circle (2pt) node{};
			\filldraw (5, -6) circle (2pt) node{};
			\filldraw (7, -6) circle (2pt) node{};
			\draw (-7, -6) -- (-7.5, -8);
			\draw (1, -6) -- (0.5, -8);
			\draw (-3, -6) -- (-3.5, -8);
			\draw (5, -6) -- (4.5, -8);
			\draw (-5, -6) -- (-5.5, -8);

			\filldraw (-7.5, -8) circle (2pt) node[below = 0.2cm] {\huge 1};
			\filldraw (0.5, -8) circle (2pt) node[below = 0.2cm] {\huge 2};			
			\filldraw (-3.5, -8) circle (2pt) node[below = 0.2cm] {\huge 3};
			\filldraw (4.5, -8) circle (2pt) node[below = 0.2cm] {\huge 4};
			\filldraw (-5.5, -8) circle (2pt) node[below = 0.2cm] {\huge 5};

		\end{tikzpicture}
		\captionsetup{justification=centering}
		\caption{The tree $T_2(20)$}
		\label{fig:comp-tree-ex-a}	
	\end{subfigure}	
	\begin{subfigure}[t]{0.49\linewidth}
		\centering
		\begin{tikzpicture}[scale=0.49, transform shape]
			
			\filldraw (0, 0) circle (2pt) node{};
			\draw (0, 0) -- (-4.5, -2);
			\draw (0, 0) -- (0, -2);
			\draw (0, 0) -- (4.5, -2);

			\filldraw (-4.5, -2) circle (2pt) node{};
			\filldraw (0, -2) circle (2pt) node{};
			\filldraw (4.5, -2) circle (2pt) node{};
			\draw (-4.5, -2) -- (-6, -4);
			\draw (-4.5, -2) -- (-4.5, -4);
			\draw (-4.5, -2) -- (-3, -4);
			\draw (0, -2) -- (-1.5, -4);
			\draw (0, -2) -- (0, -4);
			\draw (0, -2) -- (1.5, -4);
			\draw (4.5, -2) -- (3, -4);
			\draw (4.5, -2) -- (4.5, -4);
			\draw (4.5, -2) -- (6, -4);
			
			\filldraw (-6, -4) circle (2pt) node{};
			\filldraw (-4.5, -4) circle (2pt) node{};
			\filldraw (-3, -4) circle (2pt) node{};
			\filldraw (-1.5, -4) circle (2pt) node{};
			\filldraw (0, -4) circle (2pt) node{};
			\filldraw (1.5, -4) circle (2pt) node{};
			\filldraw (3, -4) circle (2pt) node{};
			\filldraw (4.5, -4) circle (2pt) node{};
			\filldraw (6, -4) circle (2pt) node{};
			\draw (-6, -4) -- (-6.5, -6);
			\draw (-1.5, -4) -- (-2, -6);
			\draw (3, -4) -- (2.5, -6);
			\draw (-4.5, -4) -- (-5, -6);
			\draw (0, -4) -- (-0.5, -6);
			\draw (4.5, -4) -- (4, -6);
			\draw (-3, -4) -- (-3.5, -6);
			
			\filldraw (-6.5, -6) circle (2pt) node[below = 0.2cm] {\huge 1};
			\filldraw (-2, -6) circle (2pt) node[below = 0.2cm] {\huge 2};
			\filldraw (2.5, -6) circle (2pt) node[below = 0.2cm] {\huge 3};
			\filldraw (-5, -6) circle (2pt) node[below = 0.2cm] {\huge 4};
			\filldraw (-0.5, -6) circle (2pt) node[below = 0.2cm] {\huge 5};
			\filldraw (4, -6) circle (2pt) node[below = 0.2cm] {\huge 6};
			\filldraw (-3.5, -6) circle (2pt) node[below = 0.2cm] {\huge 7};

		\end{tikzpicture}
		\captionsetup{justification=centering}
		\caption{The tree $T_3(20)$}
		\label{fig:comp-tree-ex-b}	
	\end{subfigure}
	\caption{Balanced binary and ternary trees}
	\label{fig:comp-tree-ex}
\end{figure}

\begin{remark}
It is not hard to see that up to isomorphism, there is a unique balanced $b$-ary tree on $n$ vertices. For convenience, we will denote by $T_b(n)$ 
some fixed balanced $b$-ary tree with vertex set $[n]$ from the isomorphism class.
\end{remark}

The key fact we use about balanced trees is that the number $n_{v}$ of vertices in the tree rooted at $v$ only depends on $n$ and on the level of $v$, up to a small error:

\begin{lemma} \label{lem:subtree-size}
	Let $T$ be a balanced $b$-ary tree. Then for any vertex $v$ of $T$
	$$
		n_{v} = \frac{n - ((b^{\lev(v)} - 1)/(b - 1))}{b^{\lev(v)}} + \beta,
	$$ 
	where $|\beta| \leq 1$.
\end{lemma}
\begin{proof}
	If $v$ is on the last level, the approximation holds trivially. Otherwise, the number $\frac{n - ((b^{\lev(v)} - 1)/(b - 1))}{b^{\lev(v)}}$ is what we obtain if the vertices of $T$ without the top $\lev(v)$ levels are evenly divided among all $b^{\lev(v)}$ trees rooted at the same level as $v$. Balancedness of the tree ensures that the vertices are as evenly divided as possible. In particular, for every vertex $u$ at level $\lev(v)$
	the number $n_u$ is within $1$ of the above average. 
\end{proof}

\subsection{Unbounded clique-width }\label{subsec:core}
The purpose of this section is to prove that any gluing of large enough balanced binary and ternary trees has large clique-width. As large clique-width implies large treewidth, this result will imply Theorem \ref{thm:main}. 
We denote by $\cw(G)$ the clique-width of graph $G$. Since we will not actually use clique-width directly, we omit its definition for brevity.
For a permutation $\varphi \in S_n$, we denote by $G_\varphi(n)$ the union 
of $T_2(n)$ and $T_3(n)$ along $\varphi$. 
Formally, we will prove the following

\begin{restatable}{theorem}{thmcw}
	\label{thm:cw}
	For any $c > 0$, there exists $n \in \mathbb N$ such that $\cw(G_\varphi(n) ) > c$ for any $\varphi \in S_n$.
\end{restatable}

Our starting point is a result from \cite{wqo-vs-cw} that gives a lower bound for clique-width. Let $U \subseteq V(G)$ and $x, y \in U$. We say that $x$ and $y$ are \emph{$U$-similar} if their sets of neighbours outside $U$ coincide. It is not difficult to see that $U$-similarity is an equivalence relation on $U$; we denote the number of equivalence classes by $\mu_G(U)$ and define 
$$
	\mu(G) := \min\limits_{\frac{|V(G)|}{3} \leq |U| \leq \frac{2|V(G)|}{3}} \mu_G(U).
$$

\begin{lemma}[\cite{wqo-vs-cw}, Lemma~4]
	\label{prop:ndcut}
	For any graph $G$, $\mu(G) \leq \cw(G)$.
\end{lemma}  

We will apply Lemma~\ref{prop:ndcut} to bound below the clique-width of the graphs $G_\varphi(n)$. First, we observe that the vertices in $G_\varphi(n)$ have degree at most 7, which allows us to prove
the following auxiliary lemma.

\begin{lemma} \label{lem:edge-bound}
	Let $G = G_\varphi(n)$, and $U \subseteq V(G)$. We have $\frac{e_G(U)}{49} \leq \mu_G(U) \leq e_G(U) + 1$. 
\end{lemma}

\begin{proof}
	Let $x_1, \dots, x_{\mu_G(U)} \in U$ be representatives of the $U$-similar equivalence classes. Since these representatives have pairwise different neighbourhoods in $\overline{U}$, at most one of them has no neighbours in $\overline{U}$. Therefore there are at least $\mu_G(U) - 1$ edges between $U$ and $\overline{U}$, from which $\mu_G(U) \leq e_G(U) + 1$.
	
	On the other hand, note that, since degree in $G$ is bounded above by 7, there must be at least $\frac{e_G(U)}{7}$ vertices in $\overline{U}$ that are incident with at least one edge between $U$ and $\overline{U}$. 
	For the same reason, each of the representatives has at most 7 neighbours in $\overline{U}$.
	This implies that $\mu_G(U) \geq \frac{1}{7} \cdot \frac{e_G(U)}{7} = \frac{e_G(U)}{49}$,
	as claimed.
\end{proof}

Let $T_b := T_b(n)$ for $b = 2, 3$, and let $G = G_\varphi(n)$ for some $\varphi \in S_n$.
Any $U$-cut  of $G$ induces the $\varphi^{-1}(U)$-cut in $T_2$ and the $U$-cut in $T_3$.
Clearly, $\max\{ e_{T_2}(\varphi^{-1}(U)), e_{T_3}(U) \} \leq e_G(U)$.
Our aim is to bound below $e_G(U)$, for all balanced $U$-cuts of $G$. 	
At the heart of our argument lies the following idea: by allowing the $U$-cut in $G$ to have few crossing edges, we are putting restrictions on the ratio $\frac{|U|}{n}$. We show that the restrictions coming from the $\varphi^{-1}(U)$-cut in $T_2$ and those coming from the $U$-cut in $T_3$ cannot be simultaneously satisfied, provided $n$ is large enough. In particular, since the restriction will only depend on $|U|$, and since $|\varphi^{-1}(U)| = |U|$, the permutation $\varphi$ loses its importance. We will prove that, given a number $k$, there exists $n(k)$ such that any balanced $U$-cut of $G$ with $n \geq n(k)$ vertices induces at least $k$ crossing edges.

\begin{restatable}{theorem}{mainequiv}
	\label{thm:main-equiv}
	For any $k \geq 1$, there exists an integer $n(k)$ such that for any $n \geq n(k)$, any $\varphi \in S_n$, the graph $G_{\varphi}(n)$ has no balanced $U$-cut with at most $k$ crossing edges.
\end{restatable}

Theorem \ref{thm:main-equiv} implies Theorem~\ref{thm:cw}.
Indeed, Theorem~\ref{thm:main-equiv} says that for $n$ large enough any balanced $U$-cut in $G$ has at least $k$ crossing edges. This bounds below by $k$ the parameter $e(G) := \min\limits_{\frac{n}{3} \leq |U| \leq \frac{2n}{3}} e_G(U)$, and then Lemma~\ref{lem:edge-bound} and Lemma~\ref{prop:ndcut} imply Theorem~\ref{thm:cw}.

\smallskip

Let us sketch how we prove Theorem~\ref{thm:main-equiv}. 
For a set $U \subseteq V(G)$ we will denote by $r(U)$ the ratio $\frac{|U|}{n}$.	
We aim to derive estimates for $r(U)$ using the $U$-cut-sets in each of the two trees $T_2$ and $T_3$, and show that the two estimates cannot agree for sufficiently large $n$. To summarise the intuition behind our argument, let us, for a moment, imagine that the ratio $n_v/n$ for every vertex $v$ at level $i$ in $T_b$ is exactly $b^{-i}$. Then, from Lemma \ref{lm:treeCuts}, we could derive that $r(U)$ belongs to the set $[\frac{1}{3}, \frac{2}{3}] \cap \{\frac{a}{b^c} : a \in \mathbb N\}$ for an appropriate choice of $c$. Since two such sets for $b = 2$ and $b = 3$ respectively have empty intersection (as we will show in Lemma \ref{lem:distance-between-sets}), this would produce a contradiction. Of course, our assumption about the ratio $n_v/n$ is not true. However, using Lemma~\ref{lem:subtree-size} and a carefully prepared set-up, we are able to show that $r(U)$ must simultaneously be close to the two sets $[\frac{1}{3}, \frac{2}{3}] \cap \{\frac{a}{2^c} : a \in \mathbb N\}$ and
$[\frac{1}{3}, \frac{2}{3}] \cap \{\frac{a}{3^c} : a \in \mathbb N\}$ -- close enough to still yield the desired contradiction.

In what follows, we write $R_{b, c}$ for the set of \emph{$b$-adic rationals} between $\frac{1}{3}$ and $\frac{2}{3}$ with denominator at most $b^c$, i.e., $R_{b, c} := [\frac{1}{3}, \frac{2}{3}] \cap \{\frac{a}{b^c} : a \in \mathbb Z\}$.
Let $r \in \mathbb R$, and $X, Y \subseteq \mathbb R$. We denote by $d(X, Y)$ the infimum of the distances between a point $x \in X$ and a point $y \in Y$. When an argument of $d(\cdot, \cdot)$
consists of single point, we will omit $\{ \}$ in the notation. We denote by $\mathbb B(X, r)$ the set of points that are distance less than $r$ from some point in $X$.
Of central importance in our argument is a simple, yet useful lower bound on the distance between the sets $R_{2, i}, R_{3, j}$: 

\begin{lemma} \label{lem:distance-between-sets}
	$d(R_{2, i}, R_{3, j}) \geq 3^{-(i + j)}$. 
\end{lemma}

\begin{proof}
	First, note that $R_{2, i} \cap R_{3, j} = \varnothing$. Indeed, if $x = \frac{p}{2^i} = \frac{q}{3^j}$, then $p3^j = q2^i$; hence, since $2^i$ and $3^j$ are coprime, we have that $2^i$ divides $p$, which implies that $x$ is an integer. But $R_{2, i} \subseteq [\frac{1}{3}, \frac{2}{3}]$ by definition. 
	
	Now put all numbers in the two sets over the common denominator $2^i3^j$. Since they are all distinct, they must differ by at least $2^{-i}3^{-j} \geq 3^{-(i + j)}$ as claimed. 
\end{proof}

For a rooted tree $T$ and a natural number $c$, we define the \emph{layer $c$ of $T$} to be the set of edges of $T$ with one endpoint at level $c-1$ and the other endpoint at level $c$.
In the next lemma, we assume that we have a balanced $U$-cut in $T_b(n)$ with at most $k$ crossing edges, and fix a layer index $c$. We may then partition the $U$-cut-set into ``top'' edges that belong to the first $c$ layers, and the ``bottom'' edges which are all the other edges in the $U$-cut-set. Assuming there is a gap of $d$ layers between the top crossing edges and the bottom ones, and using Lemma~\ref{lm:treeCuts}, we estimate the distance from $r(U)$ to $R_{b, c}$. 

\begin{lemma} \label{lem:estimate1}
	Let $c, d, k, b \in \mathbb N$ and $b \geq 2$. Suppose that $T_b(n)$ has $\ell \geq c + d$ full levels, that $U$ is a balanced cut in $T_b(n)$ with at most $k$ crossing edges, and that 
	the layers $c+1, c+2, \ldots, c+d$ have no crossing edges. Then $r(U) \in \mathbb B\left(R_{b, c}, k \cdot \left(b^{-(c + d + 1)} + \frac{4}{n}\right)\right).$
\end{lemma}

\begin{proof}
	Let $e_1, \dots, e_{k'}$ denote the edges of the $U$-cut-set, where $k' \leq k$, and let $e_i = (u_i, v_i)$ for $i \in [k']$.  Assume, without loss of generality, that $e_1, \dots, e_s$ belong to the first $c$ layers, while $e_{s + 1}, \dots, e_{k'}$ are in layers with indices at least $c + d + 1$. 
	By Lemma~\ref{lm:treeCuts}, we have $$|U| = \mathbbm{1}_{U}(r) \cdot n + \sum_{i=1}^{k'} (-1)^{\mathbbm{1}_{U}(u_i)} \cdot n_{v_i}.$$
	
	\noindent
	We split this sum into two terms 
	
	$$
		S_1 := \mathbbm{1}_{U}(r) \cdot n + \sum_{i=1}^s (-1)^{\mathbbm{1}_{U}(u_i)} \cdot n_{v_i}
		~~~~
		\text{ and}
		~~~~
		S_2 :=  \sum_{i=s + 1}^{k'} (-1)^{\mathbbm{1}_{U}(u_i)} \cdot n_{v_i},
	$$
	and write $r_i = \frac{S_i}{n}$ for $i \in \{1,2\}$, so that $r(U) = r_1 + r_2$.
	To get an upper bound for $d(R_{b, c}, r(U))$, we will estimate $d(R_{b, c}, r_1)$ and $d(r_1, r(U))$ separately, using Lemma~\ref{lem:subtree-size}, and then apply the triangle inequality. We have
	
	$$
		r_1 = \alpha_0 + \sum_{i=1}^s \alpha_i \frac{n_{v_i}}{n},
	$$ 	
	where $\alpha_i \in \{0, \pm 1\}$. Directly from Lemma~\ref{lem:subtree-size},
	
	$$
		\frac{n_{v_i}}{n} = \frac{1}{b^{\lev(v_i)}} - \frac{b^{\lev(v_i)} - 1}{n(b - 1)b^{\lev(v_i)}} + \varepsilon,
	$$ 
	where $|\varepsilon| \leq \frac{1}{n}$. In particular, the sum of the last two terms is smaller in absolute value than $\frac{2}{n}$. Since $r_1$ is a sum of $\alpha_0 \in \{\frac{a}{b^c} : a \in \mathbb Z\}$ and at most $k$ terms each within $\frac{2}{n}$ from a number of the form $\frac{1}{b^{\lev(v)}} = \frac{b^{c - \lev(v)}}{b^c} \in \{\frac{a}{b^c} : a \in \mathbb Z\}$, we get 
	$$
		d(R_{b, c}, r_1) < \frac{2k}{n}.
	$$
	
	To estimate $d(r_1, r(U))$, we observe that it is equal to $|r_2|$, and note that $S_2$ is a sum of the values $n_v$ for some vertices $v$ lying at levels below $c + d$. Hence, another application of Lemma~\ref{lem:subtree-size} and a similar calculation as the one above then give
	$$
		|r_2| < k \cdot \left(\frac{1}{b^{c + d + 1}} + \frac{2}{n}\right),
	$$ 
	and an application of the triangle inequality finishes the proof of the claim.
	
\end{proof}

Before we proceed to the proof of Theorem \ref{thm:main-equiv} we need a small auxiliary lemma.

\begin{lemma} \label{lem:big-gap}
	Let $A, M, K \in \mathbb{N}^+$ and $M \geq 2$. Suppose $L > A \cdot M^{K + 2}$, and assume that at most $K$ elements of $[L]$ are coloured red and the rest are black. Then there exists $C$ with $A < C < L/M$ such that the $MC$ elements after it are all coloured black. 
\end{lemma}

\begin{proof}
	For $0 \leq i \leq K + 1$,
	write $X_i := \{A \cdot M^i + 1, \dots, A \cdot M^{i + 1}\}$. At least one of the intervals $X_1, \dots, X_{K + 1}$ has no red elements. Let $j \geq 1$ be the smallest index such that $X_{j}$ has no red elements. Let $C$ be the largest element in $X_{j - 1}$. Then $C$ satisfies the statement of the lemma.
\end{proof}

We are now ready to prove Theorem \ref{thm:main-equiv}.

\mainequiv*
\begin{proof}
	Let $k \geq 1$ and $\ell > \lceil\log_3 2k\rceil \cdot 4^{2k+2}$. 
	We define $n(k) = 3^{5\ell/4+1}$ and let $n \geq n(k)$. 
	We fix an arbitrary $\varphi \in S_n$ and let $G = G_{\varphi}(n)$.
	Observe that the choice of $n$ guarantees that all levels up to and including level $5\ell/4$ in both $T_2(n)$ and $T_3(n)$ are filled.
	
	Assume that there exists a vertex set $U \subseteq V(G)$ such that the 
	$U$-cut is balanced and has at most $k$ crossing edges.
	Notice that the $U$-cut-set in $G$ induces the $\varphi^{-1}(U)$-cut-set in $T_2(n)$ and the $U$-cut-set in $T_3(n)$, each with at most $k$ edges.
	An edge $e$ in the $U$-cut-set of $G$ comes either from $T_2(n)$ (if $\varphi^{-1}(e)$ is an edge in $T_2(n)$), or from $T_3(n)$ (if $e$ is an edge in $T_3(n)$), or from both. Let $R(e)$ be the set
	that consists of the index of the layer of $T_2(n)$ that contains $\varphi^{-1}(e)$, if $\varphi^{-1}(e)$ is an edge in $T_2(n)$, and the index of the layer of $T_3(n)$ that contains $e$,
	if $e$ is and edge in $T_3(n)$. 
	Let $R = \bigcup R(e)$, where the union is over all edges $e$ in the $U$-cut-set of $G$.
	Clearly $R$ has at most $2k$ elements. By coloring red the numbers in $R \cap [\ell]$
	and black the numbers in $[\ell] \setminus R$ and applying Lemma \ref{lem:big-gap}
	with $A = \lceil\log_3 2k\rceil$, $M = 4$, and $K = 2k$,  we conclude that there is a $c$
	with $\lceil\log_3 2k\rceil < c < \ell/4$ such that
	the layers $c+1, c+2, \ldots, 5c$ of $T_2(n)$ contain no $\varphi^{-1}(U)$-cut-set edges, and
	the layers $c+1, c+2, \ldots, 5c$ of $T_3(n)$ contain no $U$-cut-set edges.
	
	Since $|\varphi^{-1}(U)| = |U|$ and hence $r(\varphi^{-1}(U)) = r(U)$,
	by applying Lemma \ref{lem:estimate1} to $T_2(n)$ and $T_3(n)$, we derive that 
	on the one hand
	$r(U) \in \mathbb B\left(R_{2, c}, k \cdot \left(2^{-(5c + 1)} + \frac{4}{n}\right)\right)$,
	and on the other hand
	$r(U) \in \mathbb B\left(R_{3, c}, k \cdot \left(3^{-(5c + 1)} + \frac{4}{n}\right)\right)$.
	Therefore, by the triangle inequality, we obtain an upper bound on the distance between $R_{2, c}$ and $R_{3, c}$: 
	$$
	d(R_{2, c}, R_{3, c}) \leq k \cdot \left(2^{-(5c + 1)} + \frac{4}{n}\right) + k \cdot \left(3^{-(5c + 1)} + \frac{4}{n}\right) < k \cdot 2^{-5c} + \frac{8k}{n}.
	$$ 
	
	\noindent
	Next, we will show that each of the latter two summands is smaller than $3^{-2c}/2$. 
	This will imply a contradiction to Lemma \ref{lem:distance-between-sets} and thus prove the theorem.
	We start with the first summand. Since $c > \log_3{2k}$, we have that $k < 3^c/2$, and therefore
	$$
	k \cdot 2^{-5c} < \frac{3^c \cdot 2^{-5c}}{2}  <  \frac{3^c \cdot 3^{-3c}}{2} = \frac{3^{-2c}}{2}.
	$$
	
	\noindent
	To bound the second summand, we recall that $n \geq n(k) = 3^{5\ell/4+1} > 3^{5\ell/4}$ and $\log_3 2k < c < \ell/4$.
	The latter implies that $k < \frac{3^{\ell/4}}{2}$, and therefore we have
	$$
		\frac{8k}{n} < \frac{4 \cdot 3^{\ell/4}}{3^{5\ell/4}} = \frac{4}{3^{\ell}} < \frac{4}{3^{4c}} 
		= \frac{4}{3^{2c}} \cdot 3^{-2c} < \frac{3^{-2c}}{2},
	$$
	where the last inequality follows from the fact that $c > 1$, as  $\lceil\log_3 2k\rceil < c$ and $k \geq 1$.
\end{proof}

As discussed earlier, Theorem~\ref{thm:main-equiv} together with Lemma~\ref{lem:edge-bound} and Lemma~\ref{prop:ndcut} imply Theorem~\ref{thm:cw}. 
Now, since clique-width of any graph of treewidth $k$ is at most $3 \cdot 2^{k-1}$ \cite{CR05}, Theorem \ref{thm:cw} implies Theorem~\ref{thm:main}.

\section{Positive results: achieving boundedness} \label{sec:positive}

In this section, we show that any pair of $n$-vertex graphs such that one has treewidth at most $k$ and the other pathwidth at most $\ell$ can be glued together so that the resulting graph has treewidth at most $k+3\ell+1$. This implies that we can glue together a tree and a caterpillar to a graph of treewidth at most 5. After this, we study some more restrictive conditions which allow us to guarantee tighter bounds on the treewidth of unions.
We start with some definitions that we will require.

	
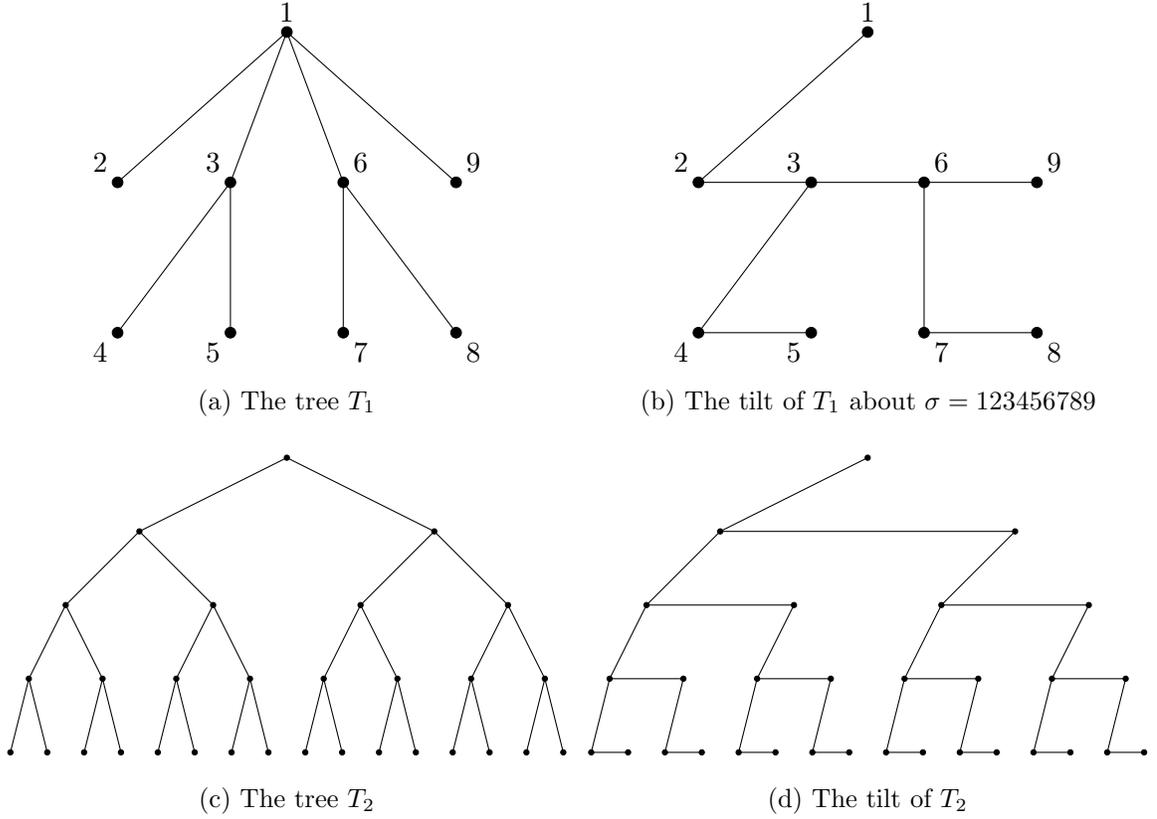
\begin{figure}[ht]
	\centering
	\begin{subfigure}[t]{0.49\linewidth}
		\centering
		\begin{tikzpicture}[scale=1, transform shape]
			
			\draw
			(0, 4) -- (-2.25,2)
			(0, 4) -- (-0.75,2)
			(0, 4) -- (0.75,2)
			(0, 4) -- (2.25,2)
			(-0.75,2) -- (-2.25,0)
			(-0.75,2) -- (-0.75,0)
			(0.75,2) -- (0.75,0)
			(0.75,2) -- (2.25,0);
			
			\filldraw 
			(0,4) circle (2pt) node[above] {$1$}
			(-2.25,2) circle (2pt) node[above left] {$2$}
			(-0.75,2) circle (2pt) node[above left] {$3$}
			(0.75,2) circle (2pt) node[above right] {$6$}
			(2.25,2) circle (2pt) node[above right] {$9$}
			(-2.25,0) circle (2pt) node[below left] {$4$}
			(-0.75,0) circle (2pt) node[below left] {$5$}
			(0.75,0) circle (2pt) node[below right] {$7$}
			(2.25,0) circle (2pt) node[below right] {$8$};

		\end{tikzpicture}
		\captionsetup{justification=centering}
		\caption{The tree $T_1$}
		\label{middlepoint}	
	\end{subfigure}	
	\begin{subfigure}[t]{0.49\linewidth}
		\centering
		\begin{tikzpicture}[scale=1, transform shape]
			
			\draw
			(0, 4) -- (-2.25,2)
			(-2.25, 2) -- (2.25,2)
			(-0.75,2) -- (-2.25,0)
			(-2.25,0) -- (-0.75,0)
			(0.75,2) -- (0.75,0)
			(0.75,0) -- (2.25,0);
			
			\filldraw 
			(0,4) circle (2pt) node[above] {$1$}
			(-2.25,2) circle (2pt) node[above left] {$2$}
			(-0.75,2) circle (2pt) node[above left] {$3$}
			(0.75,2) circle (2pt) node[above right] {$6$}
			(2.25,2) circle (2pt) node[above right] {$9$}
			(-2.25,0) circle (2pt) node[below left] {$4$}
			(-0.75,0) circle (2pt) node[below left] {$5$}
			(0.75,0) circle (2pt) node[below right] {$7$}
			(2.25,0) circle (2pt) node[below right] {$8$};

		\end{tikzpicture}
		\captionsetup{justification=centering}
		\caption{The tilt of $T_1$ about $\sigma = 123456789$}
		\label{middlepointareas}	
	\end{subfigure}
	
	\bigskip
	
	\begin{subfigure}[t]{0.49\linewidth}
		\centering
		\begin{tikzpicture}[scale=0.49, transform shape]
			
			\filldraw (0, 0) circle (2pt) node{};
			\draw (0, 0) -- (-4, -2);
			\draw (0, 0) -- (4, -2);
			
			\filldraw (-4, -2) circle (2pt) node{};
			\filldraw (4, -2) circle (2pt) node{};
			\draw (-4, -2) -- (-6, -4);
			\draw (-4, -2) -- (-2, -4);
			\draw (4, -2) -- (2, -4);
			\draw (4, -2) -- (6, -4);
			
			\filldraw (-6, -4) circle (2pt) node{};
			\filldraw (-2, -4) circle (2pt) node{};
			\filldraw (2, -4) circle (2pt) node{};
			\filldraw (6, -4) circle (2pt) node{};
			\draw (-7, -6) -- (-6, -4);
			\draw (-5, -6) -- (-6, -4);
			\draw (-3, -6) -- (-2, -4);
			\draw (-1, -6) -- (-2, -4);
			\draw (1, -6) -- (2, -4);
			\draw (3, -6) -- (2, -4);
			\draw (5, -6) -- (6, -4);
			\draw (7, -6) -- (6, -4);
			
			\filldraw (-7, -6) circle (2pt) node{};
			\filldraw (-5, -6) circle (2pt) node{};
			\filldraw (-3, -6) circle (2pt) node{};
			\filldraw (-1, -6) circle (2pt) node{};
			\filldraw (1, -6) circle (2pt) node{};
			\filldraw (3, -6) circle (2pt) node{};
			\filldraw (5, -6) circle (2pt) node{};
			\filldraw (7, -6) circle (2pt) node{};
			\draw (-7, -6) -- (-7.5, -8);
			\draw (-7, -6) -- (-6.5, -8);
			\draw (-5, -6) -- (-5.5, -8);
			\draw (-5, -6) -- (-4.5, -8);
			\draw (-3, -6) -- (-3.5, -8);
			\draw (-3, -6) -- (-2.5, -8);
			\draw (-1, -6) -- (-1.5, -8);
			\draw (-1, -6) -- (-0.5, -8);
			\draw (1, -6) -- (0.5, -8);
			\draw (1, -6) -- (1.5, -8);
			\draw (3, -6) -- (2.5, -8);
			\draw (3, -6) -- (3.5, -8);
			\draw (5, -6) -- (4.5, -8);
			\draw (5, -6) -- (5.5, -8);
			\draw (7, -6) -- (6.5, -8);
			\draw (7, -6) -- (7.5, -8);

			\filldraw (-7.5, -8) circle (2pt) node[below = 0.2cm] {};
			\filldraw (-6.5, -8) circle (2pt) node[below = 0.2cm] {};
			\filldraw (-5.5, -8) circle (2pt) node[below = 0.2cm] {};
			\filldraw (-4.5, -8) circle (2pt) node[below = 0.2cm] {};
			\filldraw (-3.5, -8) circle (2pt) node[below = 0.2cm] {};
			\filldraw (-2.5, -8) circle (2pt) node[below = 0.2cm] {};
			\filldraw (-1.5, -8) circle (2pt) node[below = 0.2cm] {};
			\filldraw (-0.5, -8) circle (2pt) node[below = 0.2cm] {};
			\filldraw (0.5, -8) circle (2pt) node[below = 0.2cm] {};
			\filldraw (1.5, -8) circle (2pt) node[below = 0.2cm] {};
			\filldraw (2.5, -8) circle (2pt) node[below = 0.2cm] {};
			\filldraw (3.5, -8) circle (2pt) node[below = 0.2cm] {};
			\filldraw (4.5, -8) circle (2pt) node[below = 0.2cm] {};			
			\filldraw (5.5, -8) circle (2pt) node[below = 0.2cm] {};
			\filldraw (6.5, -8) circle (2pt) node[below = 0.2cm] {};
			\filldraw (7.5, -8) circle (2pt) node[below = 0.2cm] {};
		\end{tikzpicture}
		\captionsetup{justification=centering}
		\caption{The tree $T_2$}
		\label{middlepoint}	
	\end{subfigure}
	\begin{subfigure}[t]{0.49\linewidth}
		\centering
		\begin{tikzpicture}[scale=0.49, transform shape]
			
			\filldraw (0, 0) circle (2pt) node{};
			\draw (0, 0) -- (-4, -2);
			\draw (-4, -2) -- (4, -2);
			
			\filldraw (-4, -2) circle (2pt) node{};
			\filldraw (4, -2) circle (2pt) node{};
			\draw (-4, -2) -- (-6, -4);
			\draw (-6, -4) -- (-2, -4);
			\draw (4, -2) -- (2, -4);
			\draw (2, -4) -- (6, -4);
			
			\filldraw (-6, -4) circle (2pt) node{};
			\filldraw (-2, -4) circle (2pt) node{};
			\filldraw (2, -4) circle (2pt) node{};
			\filldraw (6, -4) circle (2pt) node{};
			\draw (-7, -6) -- (-6, -4);
			\draw (-5, -6) -- (-7, -6);
			\draw (-3, -6) -- (-2, -4);
			\draw (-1, -6) -- (-3, -6);
			\draw (1, -6) -- (2, -4);
			\draw (3, -6) -- (1, -6);
			\draw (5, -6) -- (6, -4);
			\draw (7, -6) -- (5, -6);
			
			\filldraw (-7, -6) circle (2pt) node{};
			\filldraw (-5, -6) circle (2pt) node{};
			\filldraw (-3, -6) circle (2pt) node{};
			\filldraw (-1, -6) circle (2pt) node{};
			\filldraw (1, -6) circle (2pt) node{};
			\filldraw (3, -6) circle (2pt) node{};
			\filldraw (5, -6) circle (2pt) node{};
			\filldraw (7, -6) circle (2pt) node{};
			\draw (-7, -6) -- (-7.5, -8);
			\draw (-7.5, -8) -- (-6.5, -8);
			\draw (-5, -6) -- (-5.5, -8);
			\draw (-5.5, -8) -- (-4.5, -8);
			\draw (-3, -6) -- (-3.5, -8);
			\draw (-3.5, -8) -- (-2.5, -8);
			\draw (-1, -6) -- (-1.5, -8);
			\draw (-1.5, -8) -- (-0.5, -8);
			\draw (1, -6) -- (0.5, -8);
			\draw (0.5, -8) -- (1.5, -8);
			\draw (3, -6) -- (2.5, -8);
			\draw (2.5, -8) -- (3.5, -8);
			\draw (5, -6) -- (4.5, -8);
			\draw (4.5, -8) -- (5.5, -8);
			\draw (7, -6) -- (6.5, -8);
			\draw (6.5, -8) -- (7.5, -8);

			\filldraw (-7.5, -8) circle (2pt) node[below = 0.2cm] {};
			\filldraw (-6.5, -8) circle (2pt) node[below = 0.2cm] {};
			\filldraw (-5.5, -8) circle (2pt) node[below = 0.2cm] {};
			\filldraw (-4.5, -8) circle (2pt) node[below = 0.2cm] {};
			\filldraw (-3.5, -8) circle (2pt) node[below = 0.2cm] {};
			\filldraw (-2.5, -8) circle (2pt) node[below = 0.2cm] {};
			\filldraw (-1.5, -8) circle (2pt) node[below = 0.2cm] {};
			\filldraw (-0.5, -8) circle (2pt) node[below = 0.2cm] {};
			\filldraw (0.5, -8) circle (2pt) node[below = 0.2cm] {};
			\filldraw (1.5, -8) circle (2pt) node[below = 0.2cm] {};
			\filldraw (2.5, -8) circle (2pt) node[below = 0.2cm] {};
			\filldraw (3.5, -8) circle (2pt) node[below = 0.2cm] {};
			\filldraw (4.5, -8) circle (2pt) node[below = 0.2cm] {};			
			\filldraw (5.5, -8) circle (2pt) node[below = 0.2cm] {};
			\filldraw (6.5, -8) circle (2pt) node[below = 0.2cm] {};
			\filldraw (7.5, -8) circle (2pt) node[below = 0.2cm] {};
			
		\end{tikzpicture}
		\captionsetup{justification=centering}
		\caption{The tilt of $T_2$}
		\label{middlepoint2}	
	\end{subfigure}		
	
	\caption{An illustration of two trees and their tilts}
	\label{fig:tilt}
	
\end{figure}

In our proof we will use so-called smooth tree decompositions.
A \emph{smooth} tree decomposition    of width~$k$ is a tree decomposition in which
all bags have $k+1$ vertices and adjacent bags share exactly~$k$ vertices.
It is known that every $n$-vertex graph of treewidth $k$ has a smooth tree decomposition of width~$k$; and any smooth tree decomposition has exactly $n-k$ nodes~\cite{B96}.

We will also need to define a \emph{tilt} of a  tree, and for this we briefly discuss the conventions we use for depth-first searches.
For our purposes, the standard depth-first search tree traversal algorithm produces two outputs: a {\em DFS preordering} -- the linear ordering of the vertices which records the order in which they were first visited by the algorithm; and a {\em DFS walk} -- the sequence of vertices visited by the algorithm, with repetitions, that starts and ends in the root. By convention, we will assume that DFS searches always start at the root vertex.

Let $T$ be an $n$-vertex tree rooted at $v_1$, and let $\sigma = (v_1, \dots, v_n)$ 
be a DFS preordering of $V(T)$. 
The {\em tilt of $T$ about $\sigma$} is the tree rooted at $v_1$ that is obtained from $T$ as follows.
For every non-root vertex $x$, if $x$ is not the leftmost (with respect to $\sigma$) child of its parent $y$, then
we remove the edge $\{y,x\}$ and add the edge $\{x',x\}$, where $x'$ is the child of $y$ that 
immediately precedes $x$. For an illustration, see Figure~\ref{fig:tilt}, where we assume that the root vertex is at the top and the DFS algorithm visits children from left to right.

Tilts have some properties that will be useful to us. Let $T'$ be the tilt of $T$ about some DFS preordering $\sigma$.
Then every vertex of $T'$ has degree at most 3. Moreover, a vertex together with its children in $T$ induce a path in $T'$. Finally, we note that $T'$ admits a DFS traversal with the DFS preordering $\sigma$. 

\twpw*
\begin{proof}
	We first relabel the vertices of $G_1$. Let $(\cT,(X_y)_{y\in V(\cT)})$ be a smooth tree decomposition of $G_1$ of width $k$.
	We set an arbitrary node $r$ of $\cT$ to be its root, and
	fix a DFS preordering $\sigma = (a_1=r, a_2, \ldots, a_{n-k})$ of $V(\cT)$.
	Next, we assign unique labels from $[k+1]$ to the $k+1$ vertices in bag $X_{a_1}$ in an arbitrary
	way, and assign label $k+i$ to the  unique vertex in $X_{a_i} \setminus X_{p_i}$ for every $2 \leq i \leq n-k$,
	where $p_i$ is the parent of $a_i$ in $\cT$.
	
	We now relabel the vertices of $G_2$. For this we fix an ordering $\pi$ of $V(G_2)$ that witnesses
	the graph's separation number and assign increasing labels from $[n]$ according to $\pi$.
	In other words, for every $u,v \in V(G_2)$ we have $u < v$ if and only if $\pi(u) < \pi(v)$.
	Let $H$ be the union of $G_1$ and $G_2$ along the identity permutation. In the rest of the proof
	we construct a tree decomposition of $H$ of width at most $k+3\ell+1$.
	
	For every vertex $v \in V(G_2)$, we let 
	$$
	S(v)=\{w\in V(G_2): \exists x\in N_{G_2}(w) \mbox{ such that } \pi(w) < \pi(v)\le \pi(x) \}.
	$$ 
	By definition, we have that $\max_{v\in V(G_2)}|S(v)|$ is the separation number of $G_2$, and thus, by Theorem~\ref{thm:pathwidth-vs-number}, we have $|S(v)|\le \ell$ for every $v\in V(G_2)$.
	We observe that the definition of $S(v)$ implies the following

	\medskip
	\noindent
	\textbf{Claim 1.} \emph{$S(i+1) \subseteq \{i\} \cup S(i)$ for every $1 \leq i \leq n-1$}.	
	\medskip

	Let $(\cT^*, (Z_y)_{y\in V(\cT^*)})$ be such that the tree $\cT^*$ is the tilt of $\cT$ about $\sigma$,
	and $Z_{a_1} = X_{a_1}$ and $Z_{a_i} = X_{a_i} \cup X_{p_i}$ for every $2 \leq i \leq n-k$.

	\medskip
	\noindent
	\textbf{Claim 2.} \emph{$(\cT^*, (Z_y)_{y\in V(\cT^*)})$ is a tree decomposition of $G_1$ of width at most $k+1$.}

	\smallskip
	\noindent
	\textit{Proof.}  
	First, it is not hard to verify that if in a tree decomposition we add to a bag all vertices
	of its parent bag, then all three properties of tree decompositions are preserved.
	Therefore, $(\cT, (Z_y)_{y\in V(\cT)})$ is a tree decomposition of $G_1$.
	
	Next, we argue that replacing $\cT$ with $\cT^*$ in the tree decomposition $(\cT, (Z_y)_{y\in V(\cT)})$
	also preserves the properties. Clearly, properties (I) and (II) are preserved as we do not change
	the bags, so we only need to show that for every $v \in V(G_1)$ the subgraph $\cT^*_v$
	is connected. Since $\cT_v$ is connected and $V(\cT_v) = V(\cT^*_v)$, 
	it is enough to show that for every 
	every $\{a_i,a_j\} \in E(\cT_v)$ there is a path from $a_i$ to $a_j$ in $\cT_v^*$.
	The latter means that $v$ belongs to $Z_{a_r}$ for every $a_r$ on the path from $a_i$ to $a_j$ in $\cT^*$.
	Without loss of generality, assume that $a_i$ is the parent of $a_j$ in $\cT$, and
	let $a_{s_1}, a_{s_2}, \ldots, a_{s_t}$ be the children of $a_i$ in $\cT$ that precede $a_j$
	in $\sigma$. Then $a_i, a_{s_1}, a_{s_2}, \ldots, a_{s_t}, a_j$ is the path from $a_i$ to $a_j$
	in $\cT^*$. Now, since $v \in Z_{a_i} \cap Z_{a_j} = X_{a_i}$ and $X_{a_i}$ is a subset of all $Z_{a_{s_1}}, Z_{a_{s_2}}, \ldots, Z_{a_{s_t}}$, we conclude that $v$ belongs to all these bags, as required.
	
	To finish the proof of the claim, we recall that $(\cT,(X_y)_{y\in V(\cT)})$ is a smooth tree decomposition. 
	Hence $|X_{p_i} \setminus X_{a_i}| = 1$ for every $2 \leq i \leq n-k$, and therefore $|Z_{a_i}| \leq k+2$ for every $i \in [n-k]$, i.e. the width of $(\cT^*, (Z_y)_{y\in V(\cT^*)})$ is at most $k+1$. \qed

	\medskip
	
	Next, we will iteratively extend the bags of $(\cT^*, (Z_y)_{y\in V(\cT^*)})$ in such a way
	that it satisfies properties (I) and (III) of tree decompositions after every iteration, and it also 
	satisfies property (II) with respect to graph $H$ after the final iteration.
	Note that because the vertex set of $G_1$ and $H$ is the same and we will only extend the bags, property (I) will always hold.
 	
 	If $k+1 < \ell$ we set $r = \ell - k$, otherwise we set $r = 1$.
 	Note that $\cup_{i=1}^r Z_{a_i} = \{1, 2, \ldots, k+r\}$. 
 	In the first iteration, for every $i \in [r]$ we extend $Z_{a_i}$ to be equal $\{1, 2, \ldots, k+r \}$.
 	Since $\sigma$ is a DFS preordering for $\cT^*$, the subgraph $\cT^*[\{a_1, \ldots, a_r \}]$
 	is connected and therefore this extension preserves property (III).
 	Observe that for every $i \in [r]$, the bag $Z_{a_i}$ contains $k+i$ and $S(k+i) \subseteq Z_{a_i}$.
 	In particular, together with Claim 1 this implies that  $S(k+r+1) \subseteq \{k+r\} \cup S(k+r) \subseteq Z_{a_r}$.
 	
 	Now, for every $r+1 \leq i \leq n-k$ we perform the following iteration: we add $S(k+i)$ to every bag
 	$Z_{a_t}$, where $a_t$ is a vertex of the path from $a_{i-1}$ to $a_{i}$ in $\cT^*$.
 	We will prove by induction on $i$ that after the iteration corresponding to $i$ we have $S(k+i+1) \subseteq Z_{a_i}$,
 	and for every $v \in S(k+i)$ the subgraph $\cT^*_v$ is connected, i.e. property (III) is preserved.
 	Indeed, since after the iteration corresponding to $i$ the set $S(k+i)$ is a subset of $Z_{a_i}$ and $k+i \in Z_{a_i}$ by the initial definition of the bags,
 	we conclude from Claim 1 that $S(k+i+1) \subseteq Z_{a_i}$. To show the second part,
 	we observe that before the iteration, by the induction hypothesis, $S(k+i) \subseteq Z_{a_{i-1}}$ and for every
 	$v \in S(k+i)$ the subgraph $\cT^*_v$ is connected. Since the extension added $v$ only to the bags 
 	of the path from $a_{i-1}$ to $a_{i}$, the subgraph $\cT^*_v$ remains connected after the iteration.
 	Consequently, after all the iterations we have that $\cT^*$ satisfies properties (I) and (III) of tree decompositions
 	and also for every $i \in [n-k]$ we have that $\{k+i\} \cup S(k+i) \subseteq Z_{a_i}$.

	We will show next that the latter fact implies property (II) for $H$. From Claim 2 we already know that
	every edge of $G_1$ belongs to some bag of $\cT^*$, so it remains to show the same for every edge of $G_2$.
	Let $\{i,j\} \in E(G_2)$, where $i < j$, be an arbitrary edge of $G_2$. 
	If $j \leq k+1$, then $\{i,j\} \subseteq Z_{a_r}$. If $j > k+1$, then $\{i,j\} \subseteq Z_{j-k}$
	as $i \in S(j)$ and $\{j\} \cup S(j) \subseteq Z_{a_{j-k}}$.
  	
 	To finish the proof we observe that we can think that the bag updates (including the updates of the first $r$ bags) are done when we move from the current vertex $a_{i-1}$ of $\cT^*$ to the next unvisited vertex $a_i$ along a DFS walk. 
 	In this way, since the maximum degree of $\cT^*$ is at most 3, every vertex is visited by 
 	the DFS walk at most 3 times. Each time the corresponding bag is extended by a set of
 	size at most $\ell$. Therefore the width of the final tree decomposition is at most $k + 1 + 3 \ell$.
\end{proof}

\subsection{Tighter bounds on the treewidth of unions}

Nash-Williams' tree-covering theorem~\cite{NW64} states that a graph $G=(V,E)$ can be edge-covered by at most $t$ trees if  and only if for every non-empty set $U\subseteq V$ the number of edges in the subgraph of $G$ induced by $U$ is at most $t(|U|-1)$. 
It is not hard to see that this classical result implies that every partial 2-tree can be edge-covered by at most two trees. 
In this section we prove results that, as particular cases, give types of pairs of trees that can be glued into partial 2-trees.

We first observe that the union $H$ of a star $K_{1,n-1}$ and an $n$-vertex tree $T$ along
an arbitrary permutation $\varphi$ has treewidth at most 2.
Indeed, in $H$ all the edges that are coming from the star are incident with the
center of the star. Hence, if $\cT$ is an optimal tree decomposition of $T$ and the center of
the star is identified with vertex $i$ of $T$, then by adding $i$ to every bag of $\cT$ we obtain
a tree decomposition of $H$ of width at most 2. This argument immediately generalizes
to graphs of bounded vertex cover number and graphs of bounded treewidth.

\twvc*
\begin{proof}
	Let $G_1$ be a graph with a vertex cover $C \subseteq V(G_1)$ of size $k$,
	and let $G_2$ be a graph with a tree decomposition $\cT$ of width $t$.
	Let also $\varphi \in S_n$ be an arbitrary permutation and $H$ be the union of $G_1$ and $G_2$
	along $\varphi$.
	Then it is routine to check that by adding $\varphi(C)$ to every bag of $\mathcal{T}$ we obtain a tree decomposition of $H$ of width at most $k+t$. 
\end{proof}

The fact that the union of a star and a tree along any permutation is a partial 2-tree
distinctly sets apart our questions from the classical graph packing problems, where
unions are required to be edge-disjoint (see e.g.~\cite{HHS81, Getal17}). 
Indeed, while stars cannot be packed with any other tree, they are among the easiest trees to glue with.

We  now prove that two graphs of bounded pathwidth can always be glued into
a graph of bounded pathwidth.

\pathwidth*
\begin{proof}
	We apply Theorem~\ref{thm:pathwidth-vs-number}.	Let $\pi_{G_1}$ and $\pi_{G_2}$ be layouts of $G_1$ and $G_2$ respectively, that witness
	their separation numbers. Without loss of generality, assume that the vertices of 
	$G_1$ are labeled according to the layout $\pi_{G_1}$, and the vertices of $G_2$ are labeled according to the layout $\pi_{G_2}$. 
	Then the union of $G_1$ and $G_2$ along the identity permutation is a graph of separation number at most $k+t$, witnessed by layout $\pi_{G_1}$.
\end{proof}

Caterpillars have pathwidth 1, so the above lemma implies that two caterpillars can always be glued into a graph of pathwidth (and hence treewidth) at most 2. Note, that in contrast to stars, taking a right permutation is crucial to preserve
bounded treewidth. Indeed, two long paths can be glued together to a large square grid, 
which is known to have a large treewidth.

It is also possible to glue a path $P$ with an arbitrary tree $T$ to a graph of treewidth 
at most 2.
Informally, this can be seen as follows. Consider a planar embedding of $T$.
Then unite the vertices of $P$ with the vertices of $T$ following a depth-first search (DFS) ordering for $T$
and embed the edges of $P$ in such a way that the resulting embedding is outerplanar,
i.e. all vertices belong to the outer face.
This shows that $P$ and $T$ can be glued into an outerplanar graph, and it is known that outerplanar graphs have treewidth at most 2.
This type of argument was used in previous work on gluing paths to graphs of bounded treewidth~\cite{Q17}, and the proof of our main positive result, Theorem~\ref{th:tw-pw}, is a generalization of the argument.

The same result can also be shown using the following book embedding argument, which we use again in Section~\ref{sec:conclusion}.
A \emph{book} is a collection of half-planes, called the \emph{pages}, 
all having the same line as their boundary, which is called the \emph{spine}. 
A \emph{book embedding} of graph is a generalization of a planar embedding in which 
the vertices of the graph are mapped to the spine and the edges are embedded 
in the pages without crossings. 
The \emph{book thickness} of a graph is the smallest possible number of pages in a book embedding of the graph.
Graphs of book thickness 1 are exactly outerplanar graphs. 
In particular, any tree has book thickness 1.
To see that a path $P$ and an arbitrary tree $T$ can be glued into a graph of treewidth at
most 2 we will show that there is a gluing of book thickness 1.
Let us fix a book embedding of $T$ into a single-page book. This embedding induces a linear
order of the vertices as they appear in the spine. 
We unite $T$ with $P$ in such a way that the edges of $P$ connect consecutive vertices along the spine and therefore can be embedded in the page very close to the spine without causing
any edge intersections. Consequently, the union has book thickness 1. Therefore it is  outerplanar and hence its treewidth is at most 2.


\section{Conclusion and outlook}\label{sec:conclusion}

Our main result shows that graphs of bounded treewidth cannot always be glued into a graph of bounded treewidth, and that this is true even for graphs of treewidth 1. Yet we also showed that certain graphs of bounded treewidth
can be glued into a graph of bounded treewidth.
In particular, we observed that two caterpillars and also a path and a tree can be glued into a graph of treewidth at most 2.
We do not know if it is always possible to achieve the same for arbitrary tree and caterpillar.

\begin{question}
	Is it always possible to glue a tree and a caterpillar to a graph of treewidth at most 2?
\end{question}

\noindent
By Theorem \ref{th:tw-pw} we know that any caterpillar and any tree can be glued into a graph of treewidth at most 5. 


A tempting direction for further investigation is the gluing of three or more graphs. For the case of trees we ask two specific questions which could guide this endeavor. 


It is well-known (see, e.g., Exercise 4 of \cite[Chapter 4]{Diestel}) that since every $n$-vertex planar graph has at most $3n-6$ edges, then Nash-Williams' tree-covering theorem implies that every planar graph can be edge-covered by three trees.

\begin{question}
	Given any three $n$-vertex trees is it always possible to glue them into a planar graph? 
\end{question}

\noindent
Since every graph of book thickness 2 is planar, one  can use the book embedding argument 
from Section \ref{sec:positive} to show that two arbitrary trees and a path can be glued into a planar graph.
However, for three arbitrary trees we do not even know if they can always be
glued into a graph that excludes some fixed clique as a minor.

Gon{\c{c}}alves proved that every planar graph can be decomposed into 4 forests of caterpillars~\cite{Gon07}. 

\begin{question}
	Is it always possible to glue arbitrary 4 caterpillars into a planar graph? 
\end{question}	

\noindent
Is it possible to do this for any 3 caterpillars?
From Lemma \ref{lem:pathwidth} we know that any $k$ caterpillars can be glued into a graph of pathwidth at most $k$, and therefore into a graph that excludes some fixed clique minor.

\section*{Acknowledgements}

The third author thankfully acknowledges support from FONDECYT/ANID Iniciaci\'on en Investigaci\'on Grant 11201251, and from Programa Regional MATH-AMSUD MATH210008.



\bibliographystyle{plain}
\bibliography{references.bib}

\end{document}